\documentclass[a4paper,12pt,oneside,headsepline]{scrartcl}

\usepackage{etoolbox}
\usepackage{ifdraft}

\usepackage[utf8]{inputenc}
\usepackage[T1]{fontenc}

\usepackage{lmodern}
\usepackage{fourier}
\usepackage{amssymb, amsfonts}

\usepackage{mathrsfs} %
\usepackage{dsfont}
\usepackage[usenames,dvipsnames]{xcolor}%
\usepackage{lettrine}
\usepackage{url}

\usepackage{stmaryrd}

\usepackage[english]{babel}

\usepackage[draft=false]{scrlayer-scrpage}

\usepackage{wrapfig}
\usepackage{footmisc}
\usepackage{needspace}

\usepackage{amsmath}
\usepackage{nccmath}
\usepackage[thmmarks, amsmath, amsthm]{ntheorem}

\providebool{nobiblatex}
\boolfalse{nobiblatex}
\ifbool{nobiblatex}{%
}{%
  \usepackage[backend=biber,style=numeric-comp,giveninits,url=false,sortcites=true,maxbibnames=99]{biblatex}
  \addbibresource{bib.bib}
}

\usepackage[pdftex,final,
            pdfborder={0 0 0}, colorlinks=true,
            linkcolor=BrickRed, citecolor=ForestGreen, urlcolor=RoyalBlue]{hyperref}

\ifdraft{%
\usepackage[colorinlistoftodos]{todonotes}
\usepackage{lineno}
}{}

\usepackage{booktabs}
\usepackage[inline]{enumitem}
\usepackage{float}
\usepackage{graphicx}
\usepackage{multicol}
\usepackage{scrtime}
\usepackage{tikz}
\usetikzlibrary{arrows,calc,matrix,positioning}

\ifdraft{\date{\today}}{\date{}}

\ifdraft{%
\linenumbers
\newcommand*\patchAmsMathEnvironmentForLineno[1]{%
  \expandafter\let\csname old#1\expandafter\endcsname\csname #1\endcsname
  \expandafter\let\csname oldend#1\expandafter\endcsname\csname end#1\endcsname
  \renewenvironment{#1}%
     {\linenomath\csname old#1\endcsname}%
     {\csname oldend#1\endcsname\endlinenomath}}%
\newcommand*\patchBothAmsMathEnvironmentsForLineno[1]{%
  \patchAmsMathEnvironmentForLineno{#1}%
  \patchAmsMathEnvironmentForLineno{#1*}}%
\AtBeginDocument{%
\patchBothAmsMathEnvironmentsForLineno{equation}%
\patchBothAmsMathEnvironmentsForLineno{align}%
\patchBothAmsMathEnvironmentsForLineno{flalign}%
\patchBothAmsMathEnvironmentsForLineno{alignat}%
\patchBothAmsMathEnvironmentsForLineno{gather}%
\patchBothAmsMathEnvironmentsForLineno{multline}%
}}

\KOMAoptions{DIV=10}

\clearscrheadings
\automark[section]{subsection}

\renewcommand{\subsectionmark}[1]{}

\lohead{{\small \headertitle}}
\rohead{{\small \headerauthors}}

\RedeclareSectionCommand[%
  font=\Large\sffamily\bfseries,%
  beforeskip=1\baselineskip,%
  afterskip=0.5\baselineskip,%
  indent=0em%
  ]{section}

\RedeclareSectionCommands[%
  font=\normalfont\bfseries,%
  afterskip=-1em%
  ]{subsection,subsubsection}

\RedeclareSectionCommands[%
  font=\normalfont\itshape,%
  afterskip=-1em,%
  indent=0pt,
  ]{paragraph}

\ifbool{nobiblatex}{}{%
  \AtEveryBibitem{\clearfield{doi}}
  \AtEveryBibitem{\clearfield{isbn}}
  \AtEveryBibitem{\clearfield{issn}}
  \AtEveryBibitem{\clearfield{pages}}
  \AtEveryBibitem{\clearlist{language}}

  \renewcommand*{\bibfont}{\footnotesize}
  \renewbibmacro*{in:}{}
  \DeclareFieldFormat
    [article,inbook,incollection,inproceedings,patent,thesis,unpublished]
    {title}{#1}
}

\newenvironment{enumeratearabic*}{
\begin{enumerate*}[label=(\arabic*)] %
}{
\end{enumerate*}
}

\newenvironment{enumerateroman*}{
\begin{enumerate*}[label=(\roman*)] %
}{
\end{enumerate*}
}

\numberwithin{equation}{section}

\theoremnumbering{arabic}
\newtheorem{theoremcounter}{theoremcounter}[section]
\theoremnumbering{Alph}

\theoremstyle{plain}

\newtheorem{lemma}[theoremcounter]{Lemma}

\newtheorem{proposition}[theoremcounter]{Proposition}
\newtheorem{theorem}[theoremcounter]{Theorem}

\theoremstyle{plain}

\theoremstyle{definition}

\theoremstyle{remark}

\newtheorem{remark}[theoremcounter]{Remark}

\theoremstyle{nonumberremark}

\newtheorem{remarkcomputation}{Computation}

\ifdraft{
 \newcommand{\texpdf}[2]{#1}
}{
 \newcommand{\texpdf}[2]{\texorpdfstring{#1}{#2}}
}

\newcommand{\tx}{\ensuremath{\text}}

\newcommand{\tbf}{\bfseries}

\newcommand{\nbd}{\nobreakdash-\hspace{0pt}}

\renewcommand{\frak}{\ensuremath{\mathfrak}}

\newcommand{\frakg}{\ensuremath{\frak{g}}}

\newcommand{\rmt}{\ensuremath{\mathrm{t}}}

\newcommand{\rmA}{\ensuremath{\mathrm{A}}}

\newcommand{\rmH}{\ensuremath{\mathrm{H}}}
\newcommand{\rmI}{\ensuremath{\mathrm{I}}}

\newcommand{\rmK}{\ensuremath{\mathrm{K}}}
\newcommand{\rmL}{\ensuremath{\mathrm{L}}}
\newcommand{\rmM}{\ensuremath{\mathrm{M}}}

\newcommand{\rmR}{\ensuremath{\mathrm{R}}}

\newcommand{\rmU}{\ensuremath{\mathrm{U}}}

\newcommand{\td}{\tilde}
\newcommand{\wtd}{\widetilde}
\newcommand{\ov}{\overline}

\newcommand{\ra}{\ensuremath{\rightarrow}}

\newcommand{\lra}{\ensuremath{\longrightarrow}}

\newcommand{\mto}{\ensuremath{\mapsto}}

\newcommand{\ZZ}{\ensuremath{\mathbb{Z}}}

\newcommand{\RR}{\ensuremath{\mathbb{R}}}
\newcommand{\CC}{\ensuremath{\mathbb{C}}}

\renewcommand{\Re}{\ensuremath{\mathrm{Re}}}
\renewcommand{\Im}{\ensuremath{\mathrm{Im}}}

\newcommand{\sgn}{\ensuremath{\mathrm{sgn}}}

\newenvironment{psmatrix}{\left(\begin{smallmatrix}}{\end{smallmatrix}\right)}

\newcommand{\Mat}[1]{\ensuremath{\mathrm{Mat}_{#1}}}
\newcommand{\GL}[1]{\ensuremath{\mathrm{GL}_{#1}}}

\newcommand{\SL}[1]{\ensuremath{\mathrm{SL}_{#1}}}
\newcommand{\Mp}[1]{\ensuremath{\mathrm{Mp}_{#1}}}

\newcommand{\SO}[1]{\ensuremath{\mathrm{SO}_{#1}}}
\newcommand{\U}[1]{\ensuremath{\mathrm{U}_{#1}}}

\newcommand{\T}{\ensuremath{\rmt}}

\newcommand{\HS}{\mathbb{H}}

\newcommand{\rmIsm}{\rmI^{\mathrm{sm}}}
\newcommand{\Lie}{\mathrm{Lie}}
\newcommand{\gK}{(\frakg,K)}
\newcommand{\trace}{\mathrm{trace}}
\newcommand{\ga}{\gamma}
\newcommand{\om}{\omega}
\newcommand{\Ga}{\Gamma}

\newcommand{\quantsep}{\;{}:{}\;}

\newcommand{\R}{\mathbb{R}} %
\newcommand{\Q}{\mathbb{Q}} %
\newcommand{\C}{\mathbb{C}} %
\newcommand{\Z}{\mathbb{Z}} %
\newcommand{\N}{\mathbb{N}} %
\newcommand{\h}{\mathbb{H}} %

\newcommand{\E}{\mathcal{E}}
\newcommand{\calF}{\mathcal{F}}
\newcommand{\M}{\mathcal{M}}

\newcommand{\ISh}{\Lambda^{\mathrm{Sh}}}
\newcommand{\IM}{\Lambda^{\mathrm{M}}}

\newcommand{\headertitle}{{\normalfont%
  Harmonic weak Maa\ss\ forms of half-integral weight
}}
\newcommand{\headerauthors}{%
  C.~Alfes-Neumann,
  M.~Raum%
}
\title{%
  A classification of harmonic weak\\Maa\ss\ forms of half-integral weight
}
\author{%
Claudia Alfes-Neumann%
\thanks{The first author was supported by the Daimler and Benz Foundation and the Klaus Tschira Boost Fund.}%
\and
Martin Raum%
\thanks{The second author was partially supported by Vetenskapsr\aa det Grant~2015-04139 and~2019-03551.}%
}

\begin{document}

\thispagestyle{scrplain}
\begingroup
\deffootnote[1em]{1.5em}{1em}{\thefootnotemark}
\maketitle
\endgroup

\begin{abstract}
\noindent
{\tbf Abstract:}
We classify Harish-Chandra modules generated by the pullback to the metaplectic group of harmonic weak Maa\ss\ forms with exponential growth allowed at the cusps. This extends work by Schulze-Pillot and parallels recent work by Bringmann--Kudla, who investigated the case of integral weights. We realize each of our cases via a regularized theta lift of an integral weight harmonic weak Maa\ss\ form. Harish-Chandra modules in both integral and half-integral weight that occur need not be irreducible. Therefore, our display of the role that the theta lifting takes in this picture, we hope, contributes to an initial understanding of a theta correspondence for extensions of Harish-Chandra modules.
\\[.3\baselineskip]
\noindent
\textsf{\textbf{%
  MSC Primary:
  11F12%
}}%
\hspace{0.3em}{\tiny$\blacksquare$}\hspace{0.3em}%
\textsf{\textbf{%
  MSC Secondary:
  11F27, 11F70%
}}
\end{abstract}

\Needspace*{4em}
\addcontentsline{toc}{section}{Introduction}
\markright{Introduction}
\lettrine[lines=2,nindent=.2em]{B}{ringmann and Kudla}~\cite{bringmannkudla} provided a classification of the Harish-Chandra modules generated by the pullback to $\SL{2}(\R)$ of harmonic weak Maa\ss\ forms of integral weight with exponential growth allowed at the cusps. They complemented their classification with explicit examples for all~9 possibilities, and thus showed that these arise from harmonic weak Maa\ss\ forms. The case of half-integral weight was partially treated by Schulze-Pillot in earlier work~\cite{schulzepillot}. He restricted himself to a subclass of functions satisfying a more restrictive growth condition at the cusps. However, he did not explicitly realize the different modules that arise.

In this note we provide a classification of Harish-Chandra modules corresponding to the full class of harmonic weak Maa\ss\ forms of half-integral weight. Moreover, we explicitly realize all cases as certain theta liftings of integral weight harmonic weak Maa\ss\ forms that occur in Bringmann and Kudla's work. We therefore extend Schulze-Pillot's work in two directions. Observe that it is equally possible to realize these cases by Poincar\'e series or via an abstract cohomological argument.

Our work also provides a representation theoretic perspective on local theta liftings of harmonic weak Maa\ss\ forms. Theta liftings are an explicit realization of the theta correspondence introduced by Howe~\cite{howe} and are best understood for cusp forms. For example, they can be used to realize the correspondence of Shimura and Shintani in the classical setting of cusp forms~\cite{sh, shintani}, whose representation theoretic counterpart appears in Waldspurger's work~\cite{waldspurger}. To be able to lift forms with singularities at the cusps we consider lifts that are regularized using ideas of Harvey--Moore~\cite{harveymoore} and Borcherds~\cite{borcherds}.

Much less is known on the representation theoretic side. Kudla and Rallis analyzed invariant distributions~\cite{kudlarallis}, which arise from the Shintani lift of constants when viewed through the lens of the archimedean theta correspondence. They encountered reducible Harish-Chandra modules, as opposed to the irreducible ones that one finds when treating cusp forms. The Harish-Chandra modules that occurred in the work of Bringmann--Kudla and Schulze-Pillot and that occur in our work are generally reducible, too. In this sense, we give an initial sense of how the archimedean theta correspondence might function on reducible Harish-Chandra modules.

We illustrate our results by an example. Harish-Chandra modules in our setting can be visualized by their~$K$-type support and transitions, which reflect the behaviour of Maa\ss\ lowering and raising operators on harmonic weak Maa\ss\ forms, defined in~\eqref{eq:def:lowering_raising_operators} and the paragraph that follows it. We have the following two Harish-Chandra modules, one for~$\SL{2}(\RR)$ and the other one for~$\Mp{1}(\RR)$.
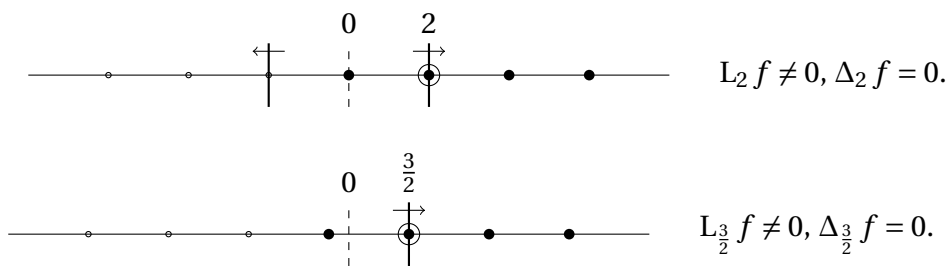
\begin{center}
\begin{tikzpicture}
\coordinate (o) at (0pt,60pt); %
\coordinate (s) at (15pt,0pt); %
\coordinate (vh) at (0pt,12pt); %

\coordinate (o) at ($(o)+0.5*(s)$);

\coordinate (lb) at ($(o)+0*(s)$); 
\coordinate (le) at ($(o)+16*(s)$);
\draw[-] (lb) -- (le);

\foreach \ix in {2,4,6}{
  \coordinate (a) at ($(o)+\ix*(s)$);
  \draw (a) circle[radius=1pt];
}

\coordinate (a) at ($(o)+8*(s)$);
\fill (a) circle[radius=2pt];

\foreach \ix in {10,12,14}{
  \coordinate (a) at ($(o)+\ix*(s)$);
  \fill (a) circle[radius=2pt];
}

\coordinate (a) at ($(o)+10*(s)$);
\draw (a) circle[radius=4pt];

\coordinate (a) at ($(o)+8*(s)$);
\coordinate (at) at ($(a)+(vh)$);
\coordinate (ab) at ($(a)-(vh)$);
\draw[dashed] (ab) -- (at) node[above] {$0$};

\coordinate (a) at ($(o)+6*(s)$);
\coordinate (at) at ($(a)+(vh)$);
\coordinate (ab) at ($(a)-(vh)$);
\draw[-,thick] (ab) -- (at);

\coordinate (b) at ($(a)+0.75*(vh)$);
\coordinate (bl) at ($(b)-0.4*(s)$);
\coordinate (br) at ($(b)+0.4*(s)$);
\draw[->] (br) -- (bl);

\coordinate (a) at ($(o)+10*(s)$);
\coordinate (at) at ($(a)+(vh)$);
\coordinate (ab) at ($(a)-(vh)$);
\draw[-,thick] (ab) -- (at) node[above] {$2$};

\coordinate (b) at ($(a)+0.75*(vh)$);
\coordinate (bl) at ($(b)-0.4*(s)$);
\coordinate (br) at ($(b)+0.4*(s)$);
\draw[->] (bl) -- (br);

\coordinate (a) at ($(o)+17*(s)$);
\node[right] at (a)
  {$\rmL_2\, f \ne 0$, $\Delta_2\, f = 0$.};

\coordinate (o) at (0pt,00pt); %
\coordinate (s) at (15pt,0pt); %
\coordinate (vh) at (0pt,12pt); %

\coordinate (lb) at ($(o)+0*(s)$); 
\coordinate (le) at ($(o)+16*(s)$);
\draw[-] (lb) -- (le);

\foreach \ix in {2,4,6}{
  \coordinate (a) at ($(o)+\ix*(s)$);
  \draw (a) circle[radius=1pt];
}

\coordinate (a) at ($(o)+8*(s)$);
\fill (a) circle[radius=2pt];

\foreach \ix in {10,12,14}{
  \coordinate (a) at ($(o)+\ix*(s)$);
  \fill (a) circle[radius=2pt];
}

\coordinate (a) at ($(o)+10*(s)$);
\draw (a) circle[radius=4pt];

\coordinate (a) at ($(o)+8.5*(s)$);
\coordinate (at) at ($(a)+(vh)$);
\coordinate (ab) at ($(a)-(vh)$);
\draw[dashed] (ab) -- (at) node[above] {$0$};

\coordinate (a) at ($(o)+10*(s)$);
\coordinate (at) at ($(a)+(vh)$);
\coordinate (ab) at ($(a)-(vh)$);
\draw[-,thick] (ab) -- (at) node[above] {$\tfrac{3}{2}$};

\coordinate (b) at ($(a)+0.75*(vh)$);
\coordinate (bl) at ($(b)-0.4*(s)$);
\coordinate (br) at ($(b)+0.4*(s)$);
\draw[->] (bl) -- (br);

\coordinate (a) at ($(o)+17*(s)$);
\node[right] at (a)
  {$\rmL_{\frac{3}{2}}\, f \ne 0$, $\Delta_\frac{3}{2}\, f = 0$.};
\end{tikzpicture}%
\end{center}

The first Harish-Chandra module corresponds to case III (b) in \cite{bringmannkudla}. It can be realized by the Eisenstein series
\begin{gather*}
  E_2^\ast(z) = 1 - 24\sum_{n \geq 1}\sigma_{1}(n)e^{2\pi i n z} \,-\, \mfrac{3}{\pi y}
\tx{.}
\end{gather*}
The second one can be realized by taking the regularized Shintani-lift of~$E_2^\ast$ (compare Section \ref{sec:lifts} for the definition). Specifically, its (twisted) Shintani-lift was computed in~\cite{alfesschwagenscheidtshintani}. We let $\Delta$ be a negative fundamental discriminant. We have
\begin{gather*}
  \sqrt{|\Delta|} \ISh_{\Delta}(E_{2}^\ast,\tau) = 12 H(|\Delta|) E_{\frac{3}{2}}^\ast(\tau)
\tx{,}
\end{gather*}
where
\begin{gather*}
  E_{\frac{3}{2}}^\ast(\tau)
=
  \sum_{D \geq 0} H(D) e^{2\pi i D \tau}
+
  \frac{1}{16\pi}
  \sum_{n \in \Z}
  v^{-\frac{1}{2}}
  \beta_{\frac{3}{2}}(4\pi n^{2}v) e^{-2\pi i n^{2}\tau}, \, v=\Im(\tau),
\end{gather*}
with~$H(0) = -\frac{1}{12}$ and~$H(D) = 0$ if~$-D \neq 0$ is not a discriminant, is Zagier's weight-$\frac{3}{2}$ Eisenstein series~\cite{zagiereisenstein}. Here, $\beta_{3/2}(s)= \int_1^\infty e^{-st} t^{-3/2}dt$.

Our work is organised as follows: We first review some necessary background on the metaplectic group and harmonic weak Maa\ss\ forms. In Section \ref{sec:gkmodules} we introduce the principal series and state the classification for the $(\mathfrak{g},K)$ modules corresponding to harmonic weak Maa\ss\ forms of half-integral weight. Then we give a short overview on Millson and Shintani theta liftings of even integral weight harmonic weak Maa\ss\ forms and close with the explicit realization of all of the modules arising from our classification.

\section*{Acknowledgements}
The authors thank Steve Kudla for sharing his ideas and Igor Burban, Jens Funke and Markus Schwagenscheidt for interesting discussions and comments. The second author thanks the Institut Mittag-Leffler, where parts of this work was conducted during the program on Moduli and Algebraic Cycles.

\section{Preliminaries}

\subsection{The metaplectic group}

We define the real metaplectic group as
\begin{gather*}
  \Mp{1}(\RR)
\;:=\;
  \big\{ (g, \om) \in \SL{2}(\CC) \,:\,
  g = \begin{psmatrix} a & b \\ c & d \end{psmatrix},
  \om :\, \HS \ra \CC \tx{\ holomorphic},\;
  \om(\tau)^2 = c \tau + d
  \big\}
\end{gather*}
equipped with the usual group law~$(g, \om) (g',\om') = (g g', \tau \mto \om(g' \tau) \om'(\tau))$. Further, we write~$\pi_{\Mp{1}}$ for the projection from~$\Mp{1}(\RR)$ to~$\SL{2}(\RR)$ that sends~$(g,\om)$ to~$g$. This turns~$\Mp{1}(\RR)$ into a connected double cover of~$\SL{2}(\RR)$.

We let~$K$, $M$, and $N$ be the preimages under~$\pi_{\Mp{1}}$ of~$\SO{2}(\RR)$, the subgroup of diagonal, and the subgroup of upper triangular unipotent matrices. We have a~$KMN$-decomposition of~$\Mp{1}(\RR)$, and the subgroups~$K$, $M$, and~$N$ are uniformized by
\begin{gather*}
  k(\theta)
:=
  \Big(
  \begin{psmatrix} \cos(\theta) & \sin(\theta) \\ -\sin(\theta) & \cos(\theta) \end{psmatrix},\,
  \omega_{k(\theta)}
  \Big)
\tx{,}\quad
  m(a, s)
:=
  \Big( \begin{psmatrix} a & 0 \\ 0 & a^{-1} \end{psmatrix}, \sqrt{a^{-1}}_s \Big)
\tx{,}
\\
  n(b)
:=
  \Big( \begin{psmatrix} 1 & b \\ 0 & 1 \end{psmatrix}, 1 \Big)
\tx{,}\quad\tx{and}\quad
  n(b) k\big(\mfrac{\pi}{2}\big)
\tx{.}
\end{gather*}
In the argument of~$k$, we have~$\theta \in \RR$ and $\om_{k(\theta)} :\, \HS \ra \CC$ is uniquely defined by its value~$\om_{k(\theta)}(i) = \exp(-i \frac{1}{2} \theta)$. To specify the right hand side of~$m(a,s)$, we define the sign function~$\sgn(i a) := i \sgn(a)$ for~$a \in i \RR$. Given~$a \in \RR$, $a > 0$, $s \in \{\pm 1\}$, or~$a \in \RR$, $a < 0$, $s \in \{\pm i\}$, we let~$\sqrt{a^{-1}}_s$ be the square root of~$a^{-1}$ with sign~$s$. The argument of~$n$ is~$b \in \RR$.

\subsection{The Lie algebra of~\texpdf{$\Mp{1}(\RR)$}{Mp_1(R)}}
\label{ssec:lie-algebra}

Since~$\pi_{\Mp{1}}$ is a covering map of Lie groups, the (complexified) Lie algebras of~$\Mp{1}(\RR)$ and~$\SL{2}(\RR)$ are canonically isomorphic. We follow the notation in~\cite{bringmannkudla}, and set
\begin{gather*}
  H
\;:=\;
  i \begin{psmatrix} 0 & -1 \\ 1 & 0  \end{psmatrix}
\tx{,}\quad
  X_+
\;:=\;
  \mfrac{1}{2} \begin{psmatrix} 1 & i \\ i & -1  \end{psmatrix}
\tx{,}\quad
  X_-
\;:=\;
  \mfrac{1}{2} \begin{psmatrix} 1 & -i \\ -i & -1  \end{psmatrix}
\tx{,}
\end{gather*}
which is a basis for~$\Lie(\Mp{1}(\RR))_\CC \cong \Lie(\SL{2}(\RR))_\CC \cong \{ A \in \Mat{2}(\CC) \,:\, \trace(A) = 0 \}$. We write~$\rmU(\Lie(\Mp{1}(\RR))_\CC)$ for the universal enveloping algebra of~$\Lie(\Mp{1}(\RR))_\CC$.

We have the commutator relations
\begin{gather*}
  \big[ X_+, X_- \big]
=
  H
\tx{,}\quad
  \big[ H, X_+ \big]
=
  2 X_+
\tx{,}\quad
  \big[ H, X_- \big]
=
  -2 X_-
\tx{.}
\end{gather*}
They allow us to verify that the Casimir element
\begin{gather}
\label{eq:def:casimir-operator}
  C
\;:=\;
  H^2 + 2 X_+ X_- +  2 X_- X_+
\in
  \rmU\big( \Lie(\Mp{1}(\RR))_\CC \big)
\end{gather}
is central as required. We have~$C = (H-1)^2 + 4 X_+ X_- - 1$ and~$C = (H+1)^2 + 4 X_- X_+ - 1$, which is slightly more convenient for later purposes.

The action of~$X + i Y \in \Lie(\Mp{1}(\RR))_\CC$, $X, Y \in \Lie(\Mp{1}(\RR))$, on smooth complex functions~$\td{f} :\, \Mp{1}(\RR) \ra \CC$ is defined by
\begin{gather*}
  \big( (X + i Y) \td{f} \big) (g)
=
  \partial_{t = 0}\, \td{f} \big(g \exp(t X)\big)
  \,+\,
  i
  \partial_{t = 0}\, \td{f} \big(g \exp(t Y)\big)
\tx{,}
\end{gather*}
where we write~$\partial_{t=0}$ for the value at~$t = 0$ of the derivative with respect to~$t$. We remark that the action of~$\rmU(\Lie(\Mp{1}(\RR))_\CC)$ is linear in a similar way, but it does not arise from composition of differential operators. This becomes relevant, when computing the action of the Casimir element.

\subsection{Harmonic weak Maa\ss\ forms}

The action of~$\SL{2}(\RR)$ on the Poincar\'e upper half plane~$\HS$ extends to the metaplectic group via~$\pi_{\Mp{1}}$:
\begin{gather*}
  \big( \begin{psmatrix} a & b \\ c & d \end{psmatrix},\, \om \big)\tau
:=
  \frac{a \tau + b}{c \tau + d}
\tx{.}
\end{gather*}
We define the slash action of weight~$k \in \frac{1}{2} \ZZ$ on functions~$f :\, \HS \ra \CC$ by
\begin{gather*}
  \big( f \big|_k (\ga, \om) \big)(\tau)
:=
  \om(\tau)^{-2k}\, f(\ga \tau)
\tx{.}
\end{gather*}

An arithmetic type is a finite dimensional, complex representation~$\rho$ of a finite index subgroup~$\Ga \subseteq \Mp{1}(\ZZ)$. We write~$V(\rho)$ for the representation space of~$\rho$. Functions~$f :\, \HS \ra V(\rho)$ admit the following slash actions for~$k \in \frac{1}{2} \ZZ$:
\begin{gather*}
  \big( f \big|_{k,\rho} (\ga, \om) \big)(\tau)
:=
  \rho((\ga,\om))^{-1}\, \big( f \big|_k (\ga, \om) \big)(\tau)
\tx{.}
\end{gather*}

The Weil representation is an arithmetic type that is most relevant in the context of theta lifts. Consider a finite quadratic module~$D = (M,q)$, with corresponding bilinear form~$\langle\,\cdot,\cdot\,\rangle_q$. We let $e(x):=e^{2\pi i x}$. There is a unique representation~$\rho_D$ for which~$V(\rho_D)$ is the free module~$\CC M$ with basis~$M$ and actions
\begin{align*}
  \rho_D \Big( \big(
  \begin{psmatrix}1 & 1 \\ 0 & 1\end{psmatrix}, 1
  \big) \Big)
  m
&{}:=
  e\big( q(m) \big) \, m
\tx{,}\\
  \rho_D \Big( \big(
 \begin{psmatrix}0 & -1 \\ 1 & 0\end{psmatrix}, \sqrt{\tau}
 \big) \Big)
  \, m
&{}:=
  \frac{1}{\sigma(D) \, \sqrt{\# M}}
  \sum_{m' \in M}
  e\big(- \langle m, m'\rangle_q \big)\, m'
\tx{,}
\end{align*}
where
\begin{gather*}
  \sigma(D)
:=
  \frac{1}{\sqrt{\# M}}
  \sum_{m \in M} e\big(-q(m) \big)
\tx{.}
\end{gather*}

Recall that $\tau=u+iv$. We fix the normalization of the Maa\ss\ lowering and raising operators as
\begin{gather}
\label{eq:def:lowering_raising_operators}
  \rmR_k
\;:=\;
  2 i \partial_{\tau} + k v^{-1}
\quad\tx{and}\quad 
  \rmL_k
\;:=\;
  -2 i v^2 \partial_{\ov\tau}
\tx{.}
\end{gather}
Then the weight-$k$ Laplace operator equals~$\Delta_k := - \rmR_{k-2} \rmL_k$.

A harmonic weak Maa\ss\ form of weight~$k \in \frac{1}{2} \ZZ$ and arithmetic type~$\rho$ for~$\Ga \subseteq \Mp{1}(\ZZ)$ is a smooth function~$f :\, \HS \ra V(\rho)$ with~$\Delta_k\, f = 0$ such that
\begin{gather*}
  \forall \ga \in \Ga \quantsep f \big|_{k,\rho}\, \ga = f
\end{gather*}
and for some norm~$\|\,\cdot\,\|$ on~$V(\rho)$
\begin{gather}
\label{eq:exp_growth_condition}
  \exists a \in \RR\,
  \forall \ga \in \Mp{1}(\ZZ) \quantsep
  \big\| \big( f \big|_k\, \ga \big)(\tau) \big\| \ll \exp(a v)
\tx{.}
\end{gather}
Note that we only have non-zero harmonic weak Maa\ss\ forms if~$\rho((-I,i)) = i^{-2k}$.

We denote the space of such forms by $\rmH_k^{\mathrm{mg}} (\Gamma,\rho)$ as in~\cite{bringmannkudla}. Observe that while~$\mathrm{mg}$ stands for moderate growth, the condition imposed in~\eqref{eq:exp_growth_condition} differs from what is called the moderate growth condition in the context of modular forms.

If $(\rho,V)$ is one-dimensional, i.e., given by a character $\chi:\Gamma\to \C^\times$, we write $\rmH_k^{\mathrm{mg}} (\Gamma,\chi)$, or even $\rmH_k^{\mathrm{mg}} (\Gamma)$ if $\chi$ is trivial. This subspace is referred to as the space of scalar-valued harmonic weak Maa\ss\ forms of weight $k$ for $\Gamma$. 

Harmonic weak Maa\ss\ forms are related to classical spaces of modular forms by the~$\xi$-operator. Following Bruinier and Funke \cite{bf2004} we define it by
\begin{gather*}
  \xi_k\,f := 2i v^k\, \ov{ \partial_{\ov\tau} f}
\tx{.}
\end{gather*}
Proceeding as in \cite{bf2004} we see that
\begin{gather*}
  \xi_k :\,
  \rmH_k^{\mathrm{mg}} (\Gamma,\rho)
\lra
  \rmM^!_{2-k} (\Gamma,\ov\rho)
\tx{,}
\end{gather*}
and that $\xi_k$ is surjective. Here, $\rmM^!_{2-k}$ denotes the subspace of weakly holomorphic modular forms (i.e., forms that are holomorphic on $\h$ and have poles of finite order at the cusps). Moreover, $\ov \rho$ is defined by $\ov\rho (\gamma)v= \overline{\rho(\gamma)\ov v}$. Here we clearly need to assume that $V$ is defined over $\R$. 

A natural subspace of $\rmH_k^{\mathrm{mg}} (\Gamma,\ov\rho)$ consists of those functions that map to cusp forms under the $\xi$-operator or alternatively for which there exists a polynomial $P_f(\tau)\in V[q^{-1}]$ such that
\begin{gather*}
  f(\tau)-P_f(\tau) \ll e^{-\varepsilon v}
\end{gather*}
as $v\to\infty$ for some $\varepsilon>0$ (and similarly at the other cusps). We denote the subspace of these forms by $\rmH_k (\Gamma,\rho)$. Its image under the~$\xi$-operator are cusp forms.

We now describe the Fourier expansion of such forms. A scalar-valued harmonic weak Maa\ss\ form of integral weight $k\neq 1$ has a Fourier expansion of the form
\begin{gather}\label{fe}
f(\tau) = f^+(\tau)+f^-(\tau) = \sum_{n\gg-\infty} c_f^+(n)q^n + c_f^-(0)v^{1-k} +\sum_{n \ll \infty} c_f^-(n) W_k(4\pi n v) q^n
\end{gather}
at $\infty$, where $W_k(x)$ is the real-valued incomplete $\Gamma$-function
\[
W_k(x)=\Re\big(\Gamma(1-k,-2x)\big)=\Gamma(1-k,-2x)+\begin{cases} \frac{(-1)^{1-k} \pi}{(k-1)!}& x>0,\\0&x<0,\end{cases}
\]
with $\Gamma(s,x)=\int_x^\infty e^{-t} t^{s-1} dt$. If $k=1$, we have to replace the term $v^{1-k}$ in the non-holomorphic part by $-\log(v)$.

If $f\in \rmH_k(\Gamma)$, then we have $c^-_f(n)=0$ for all $n\geq 0$, and more generally~$c^-_{f |_k \ga}(n) = 0$ for all~$\ga \in \Mp{1}(\ZZ)$ and~$n \ge 0$.

We now let $k \in \ZZ_{\ge 0}$. A further differential operator which establishes relations between harmonic weak Maa\ss\ forms and classical modular forms is
\[
D := \frac{1}{2\pi i } \partial_\tau.
\]
By Bol's identity we have
$
D^{1-k}= (-4\pi)^{k-1} \rmR_k^{1-k}
$
and $D^{1-k}: \rmH_k^{\mathrm{mg}}(\Gamma,\rho)\to \rmM^!_{2-k}(\Gamma,\rho)$. 

For scalar-valued forms we define the \emph{flipped space} by
\[
\rmH_k^{\#} (\Gamma) :=\{ f \in \rmH_k^{\mathrm{mg}}(\Gamma)\,:\, D^{1-k}(f)\in S_{2-k}(\Gamma)\}= \{  f \in \rmH_k^{\mathrm{mg}}(\Gamma)\,:\,  c^+_f(n)=0 \text{ for } n<0\}.
\]
The spaces $\rmH_k(\Gamma)$ and $\rmH_k^{\#}(\Gamma)$ are ``flipped'' by the \emph{flipping operator}
\[
\mathcal{F}_k:=\frac{v^{-k}}{(-k)!} \overline{\rmR_k^{-k}}.
\]
The flipping operator satisfies 
\[
\mathcal{F}_k\circ \mathcal{F}_k (f)=f.
\]
It acts on the Fourier expansion \eqref{fe}
of a harmonic weak Maa\ss\ form $f\in \rmH^{\mathrm{mg}}(\Gamma)$ by
\begin{gather}
\label{eq:feflip}
\begin{aligned}
\mathcal{F}_k(f(\tau))&= -\overline{c_f^-(0)} v^{1-k} -(-k)! \sum_{\substack{n\gg-\infty\\n\neq 0}} \overline{c_f^-(-n)}q^n -\overline{c_f^+(0)} \\
&\quad -\frac{1}{(-k)!}\sum_{\substack{n\ll\infty\\n\neq 0}} \overline{c_f^+(-n)} \Gamma(1-k,-4\pi nv) q^n.
\end{aligned}
\end{gather}

\section{\texpdf{$\gK$-modules}{gK-modules} and harmonic weak Maa\ss\ forms}\label{sec:gkmodules}

Recall that a~$\gK$\nbd module is a simultaneous module for a Lie algebra~$\frakg$ and a compact group~$K$ with~$\Lie(K) \subseteq \frakg$ and suitable compatibility conditions imposed. A Harish-Chandra module is an admissible $\gK$\nbd module, i.e., a~$\gK$\nbd module with finite dimensional~$K$\nbd isotypical components. The latter are referred to as~$K$\nbd types. In this paper, we set~$\frakg = \Lie(\Mp{1}(\RR))$, and~$K = \pi_{\Mp{1}}^{-1}(\SO{2}(\RR))$ has already been defined.

\subsection{Some principal series}
\label{ssec:principal-series}

We start by providing a sufficient supply of $\gK$-mod\-ules by decomposing suitable degenerate principal series. We adapt Section~4 of~\cite{bringmannkudla} to the case of the metaplectic group~$\Mp{1}(\RR)$.

For half-integers~$\epsilon \in \frac{1}{2} \ZZ \slash 2 \ZZ$ and~$\nu \in \CC$, we let~$\rmIsm(\epsilon, \nu)$ be the principal series representation of~$\Mp{1}(\RR)$ on the space of smooth functions with the property
\begin{gather}
\label{eq:def:principal-series-functions}
  \phi\big( n(b) m(a,s) g \big)
=
  s^{2\epsilon} |a|^{v + 1}\, \phi(g)
\tx{.}
\end{gather}
To check that this space is not empty, we only need to see that the intersection of~$M$ and~$N$ is trivial.

We consider the associated~$\gK$\nbd module~$\rmI(\epsilon, \nu)$ of~$K$\nbd finite functions in~$\rmIsm(\epsilon, \nu)$. Given an half-integer~$j$ with~$j \in \epsilon + 2 \ZZ$, we let~$\phi_j \in \rmI(\epsilon, \nu)$ be the unique function that satisfies
\begin{gather*}
  \phi_j(k(\theta))
=
  \exp(i j \theta)
\tx{.}
\end{gather*}
To see that this is well-defined it suffices to check that the defining property for~$\rmI(\epsilon, \nu)$ holds on~$K \cap MN$. For~$\theta \in \pi \ZZ$, we have
\begin{gather*}
  \exp(i j \theta)
=
  \phi_j\Big(
  \begin{psmatrix} \cos(\theta) & 0 \\ 0 & \cos(\theta) \end{psmatrix},\,
  \exp\big( -i \tfrac{1}{2} \theta \big)
  \Big)
=
  \sgn\big( \exp(i \tfrac{1}{2} \theta) \big)^{2 \epsilon}\,
  \phi_j(1)
=
  \exp\big( i \epsilon \theta \big)
\tx{.}
\end{gather*}
This calculation also shows that no other~$\rmK$-types but those corresponding to the functions~$\phi_j$, $j \in \epsilon + 2 \ZZ$ can occur in~$\rmI(\epsilon,\nu)$.

We next want to decompose the~$\rmI(\epsilon, \nu)$. To this end, we first calculate that
\begin{gather}
\label{eq:phij-Xpm}
  X_\pm \,\phi_j
=
  \mfrac{1}{2} \big( \nu + 1 \pm j \big) \phi_{j \pm 2}
\quad\tx{and}\quad
  H\, \phi_j
=
  j \phi
\tx{.}
\end{gather}
Using the defining expression for the Casimir element, we conclude that
\begin{gather*}
  C\, \phi_j
\;=\;
  \big( \nu^2 - 1 \big)\, \phi_j
\tx{.}
\end{gather*}
These calculations can be performed in a neighborhood of the identity of~$\Mp{1}(\RR)$, which allows us to transfer all computations from~$\SL{2}(\RR)$. We have chosen normalizations in such a way that the formulas match the common ones, provided in~\cite{bringmannkudla}.

For later purposes, we record that for~$r \ge 0$, we have
\begin{gather}
\label{eq:phij-Xmp-Xpm}
\begin{aligned}
  X_-^r X_+^r\,\phi_j
&=
  \Big(
  \prod_{s = 0}^{r-1}
  \mfrac{1}{2}
  \big( \nu + 1 - ((j+2r) - 2s) \big)
  \Big)\,
  \Big(
  \prod_{s = 0}^{r-1}
  \mfrac{1}{2}
  \big(\nu + 1 + (j + 2s) \big) 
  \Big)\,
  \phi_j
\tx{.}
\\
  X_+^r X_-^r\,\phi_j
&=
  \Big(
  \prod_{s = 0}^{r-1}
  \mfrac{1}{2}
  \big(\nu + 1 + ((j-2r) + 2s) \big) 
  \Big)\,
  \Big(
  \prod_{s = 0}^{r-1}
  \mfrac{1}{2}
  \big( \nu + 1 - (j - 2s) \big)
  \Big)\,
  \phi_j
\tx{.}
\end{aligned}
\end{gather}

If~$\epsilon \in \ZZ$ and~$\nu \not\in \ZZ$, or if~$\epsilon \in \frac{1}{2} + \ZZ$ and~$\nu \not\in \frac{1}{2} + \ZZ$, then~$\rmI(\epsilon, \nu)$ is irreducible. The case of~$\epsilon \in \ZZ$ and~$\nu \in \ZZ$ has already been explained in previous work~\cite{bringmannkudla}. For the purpose of the present paper, we examine~$\epsilon \in \frac{1}{2} + \ZZ$ and~$\nu \in \frac{1}{2} + \ZZ$. Then there is a unique~$j \in \epsilon + 2\ZZ$ and a unique choice of sign such that~$X_\pm\, \phi_j$ vanishes.

We remark that this differs from the situation of integral~$\epsilon$, in which there are two such pairs of~$j$ and signs except for the very special case of~$\nu = 0$. If~$\epsilon \in \frac{1}{2} + \ZZ$, we distinguish two cases based on the parity of~$\nu - \epsilon$ as opposed to the sign of~$\nu$, which plays a decisive role when~$\epsilon \in \ZZ$.

We will encounter two families of irreducible representations, whose $K$-types are spanned by
\begin{gather}
  \varpi^+(\nu)
\;=\;
  \big[
  \phi_{\nu+1},\, \phi_{\nu+2},\, \phi_{\nu+3},\, \ldots
  \big]
\quad\tx{and}\quad
  \varpi^-(\nu)
\;=\;
  \big[
  \ldots,\, \phi_{-\nu-3},\, \phi_{-\nu-2},\, \phi_{-\nu-1}
  \big]
\tx{.}
\end{gather}
They lie in the discrete series if~$\nu > 0$.

Consider the case that~$\epsilon \in \frac{1}{2} + \ZZ$ and~$\nu \in \epsilon + 2 \ZZ$. Then $-\nu - 1 \in \epsilon + 2 \ZZ$, and we have~$X_+\,\phi_{-\nu-1} = 0$, which implies that~$\phi_{-\nu-1}$ spans the maximal~$K$\nbd type of a subrepresentation of~$I(\epsilon,\nu)$. We have the short exact sequence
\begin{gather*}
  0
\lra
  \varpi^-(\nu)
\lra
  I(\epsilon,\nu) 
\lra
  \varpi^+(-\nu)
\lra
  0
\tx{.}
\end{gather*}

Consider the case that~$\epsilon \in \frac{1}{2} + \ZZ$ and~$\nu + 1 \in \epsilon + 2 \ZZ$. Then we have~$X_-\,\phi_{\nu+1} = 0$, which implies that~$\phi_{\nu+1}$ spans the minimal~$K$\nbd type of a subrepresentation of~$I(\epsilon,\nu)$. We have the short exact sequence
\begin{gather*}
  0
\lra
  \varpi^+(\nu)
\lra
  I(\epsilon,\nu) 
\lra
  \varpi^-(-\nu)
\lra
  0
\tx{.}
\end{gather*}

\subsection{Harish-Chandra modules associated to harmonic weak Maa\ss\ forms}

We consider an arithmetic type~$\rho :\, \Ga \ra \GL{}(V(\rho))$, $\Ga \subset \Mp{1}(\ZZ)$, and let~$\rmA(\Mp{1}(\RR))$ be the space of complex-valued, smooth functions on~$\Mp{1}(\RR)$ that are linear combinations of functions with the property that
\begin{gather}
\label{eq:k-eigenfunction}
  \exists j \in \tfrac{1}{2} \ZZ \,
  \forall \theta \in \RR, (g,\omega) \in \Mp{1}(\RR) \quantsep
  \td{f}\big( (g,\omega) k(\theta) \big)
=
  \td{f}\big (g,\omega) \big) \exp(i j \theta) 
\tx{.}
\end{gather}
We consider~$\rmA(\Mp{1}(\RR), V(\rho)) := \rmA(\Mp{1}(\RR)) \otimes_\CC V(\rho)$, the space of smooth functions on~$\Mp{1}(\RR)$ that take values in~$V(\rho)$ and are linear combinations of functions with the same property~\eqref{eq:k-eigenfunction}. Finally, let~$\rmA(\Mp{1}(\RR), \rho) \subseteq \rmA(\Mp{1}(\RR), V(\rho))$ be the subspace of functions with the additional property that
\begin{gather}
\label{eq:rho-covariant}
  \forall \ga \in \Ga, (g,\omega) \in \Mp{1}(\RR) \quantsep
  \td{f}(\ga (g,\omega))
=
  \rho(\ga) \td{f}\big( (g,\omega) \big)
\tx{.}
\end{gather}

We can associate Harish-Chandra modules, in fact submodules of~$\rmA(\Mp{1}(\RR), \rho)$ to harmonic weak Maa\ss\ forms of half-integral weight and of type~$\rho$. We loosely follow the description of the integral weight case in~\cite{bringmannkudla}.

In the rest of this subsection we fix~$f \in \rmH^{\mathrm{mg}}_k(\Ga,\rho)$, and set
\begin{gather}
\label{eq:def:tdf}
  \td{f}_k\big( (g,\om) \big)
\;:=\;
  \td{f}\big( (g,\om) \big)
\;:=\;
  \omega(i)^{-2k}\, f(g i)
\;=\;
  \big( f \big|_k\,(g,\om) \big)(i)
\tx{.}
\end{gather}
Note that~$\td{f}$ depends on~$k$, but it is customary to suppress this dependence from the notation. As long as~$f$ transforms like a modular form, the weight~$k$ can be recovered from the asymptotic expansion of~$f \circ \ga$ for suitable~$\ga \in \Mp{1}(\ZZ)$.

We see that~$\td{f}$ is a function from~$\Mp{1}(\RR)$ to~$V(\rho)$, and verify that it satisfies~\eqref{eq:k-eigenfunction} for~$j = k$. We calculate the action of~$K$ by right shifts to find that~$\td{f} \in \rmA(\Mp{1}(\RR), V(\rho))$. A similar calculation also shows that~$H \td{f} = k \td{f}$. This merely reflects the fact that~$\td{f} = \td{f}_k$ was constructed via the weight-$k$ slash action in~\eqref{eq:def:tdf}. In particular, it does not use any modular properties of~$f$, let alone the fact that it is harmonic. Finally, inspecting the action of~$\Mp{1}(\ZZ)$ by left shifts then shows that~$\td{f} \in \rmA(\Mp{1}(\RR), \rho)$.

Calculations for~$X_+$ and~$X_-$ are significantly more involved. They yield the same results as in the integral weight case. The lowering and raising operators intertwine with the construction in~\eqref{eq:def:tdf} provided that the weight~$k$ is adjusted:
\begin{gather}
\label{eq:Xpm-RL-intertwining}
  X_+\,\td{f}_k
=
  \big(\wtd{\rmR_k\,f}\big)_{k+2}
\quad\tx{and}\quad
  X_-\,\td{f}_k
=
  \big(\wtd{\rmL_k\,f}\big)_{k-2}
\tx{.}
\end{gather}

Only now we employ the fact that~$f$ is harmonic, i.e., that we have~$\rmR_{k+2}\,\rmL_k\,f = 0$. Recalling that~$C = (H-1)^2 + 4 X_+ X_- - 1$, this corresponds via~\eqref{eq:Xpm-RL-intertwining} to
\begin{gather}
\label{eq:tdf-C-eigenfunction}
  X_+\, X_- \td{f} = 0
\quad\tx{and}\quad
  C\, \td{f} = \big( (k-1)^2 - 1 \big) \td{f}
\tx{.}
\end{gather}

The Poincaré-Birkhoff-Witt property of the generators~$H, X_\pm$ of~$\rmU(\Lie(\Mp{1}(\RR))_\CC)$ in conjunction with~\eqref{eq:tdf-C-eigenfunction} implies that~$\td{f}$ generates a $\gK$-module~$\varpi(f,k) = \varpi_\infty(f,k) \subset \rmA(\Mp{1}(\RR), \rho)$, which is spanned by the functions
\begin{gather*}
  \td{f}_{k + 2r}
\;:=\;
  X_+^r\,\td{f}
\quad\tx{and}\quad
  \td{f}_{k - 2r}
\;:=\;
  X_-^r\,\td{f}
\tx{,}\quad
  r \in \ZZ_{\ge 0}
\tx{.}
\end{gather*}
The commutation relations of~$H$ and~$X_\pm$ then imply that each~$K$-type in~$\varpi(f,k)$ occurs with multiplicity at most once. In particular, $\varpi(f,k)$ is a Harish-Chandra module.

We finish with the eigenvalues of~$\td{f}$ under~$X_-^r X_+^r$ and the eigenvalues of~$\td{f}_{k-2}$ under~$X_+^r X_-^r$. The next lemma will be helpful when identifying~$\varpi(f,k)$ in the context of our classification. It features the Pochhammer symbols
\begin{gather*}
  (k)_r
:=
  \lim_{s \ra 0} \Ga(k+r+s) \big\slash \Ga(k+s)
=
  k \cdot (k+1) \cdots (k+r-1)
\tx.
\end{gather*}

\begin{lemma}
\label{la:tdf-XmpXpm-eigenvalues}
Fix~$k \in \frac{1}{2} + \ZZ$ and let~$f :\, \HS \ra V$ be a smooth function with~$\Delta_k\, f = 0$ for some complex vector space~$V$. Then for~$\td{f}$ defined in~\eqref{eq:def:tdf}, we have
\begin{gather}
\label{eq:tdf-XmpXpm-eigenvalues}
  X_-^r X_+^r\, \td{f}_k
=
  (-1)^r r!\, (k)_r\, \td{f}_k
\quad\tx{and}\quad
  X_+^r X_-^r\, \td{f}_{k-2}
=
  r!\, (k-1-r)_r\, \td{f}_{k-2}
\tx{.}
\end{gather}
\end{lemma}
\begin{proof}
Since the Casimir element is central it acts by scalars on the module generated by~$\td{f}$. We have~$C \td{f}_{k + 2r} = ((k-1)^2 - 1) \td{f}_{k + 2r}$ for all~$r \in \ZZ$. The action of~$H$ was determined before: $H\, \td{f}_{k+2r} = (k + 2r) \td{f}_{k+2r}$.

We conclude that for~$r \ge 0$, we have
\begin{gather*}
  4 X_- X_+\, \td{f}_{k+2r}
=
  \big( (k-1)^2 - 1 - (k + 2r + 1)^2 + 1 \big)\, \td{f}_{k+2r}
=
  -4 (r+1) (k+r)\, \td{f}_{k+2r}
\tx{.}
\end{gather*}
Similarly, if~$r \ge 0$, we have
\begin{gather*}
  4 X_+ X_-\, \td{f}_{k-2r}
=
  \big( (k-1)^2 - 1 - (k - 2r - 1)^2 + 1 \big)\, \td{f}_{k-2r}
=
  4 r (k - 1 - r)\, \td{f}_{k-2r}
\tx{.}
\end{gather*}
This yields the recursions, valid for~$r > 0$,
\begin{alignat*}{2}
  X_-^r X_+^r\, \td{f}_k
&{}=
  X_-^{r-1}\, X_- X_+\, \td{f}_{k + 2r-2}
\\
&{}=
  -r (k+r-1)\, X_-^{r-1}\, \td{f}_{k + 2r-2}
&&{}=
  -r (k+r-1)\, X_-^{r-1} X_+^{r-1}\, \td{f}_k
\tx{,}
\\
  X_+^r X_-^r\, \td{f}_{k-2}
&{}=
  X_-^{r-1}\, X_+ X_-\, \td{f}_{k-2r}
\\
&{}=
  r (k-1 - r)\, X_+^{r-1}\, \td{f}_{k-2r}
&&{}=
  r (k-1 - r)\, X_+^{r-1} X_-^{r-1}\, \td{f}_{k-2}
\tx{.}
\end{alignat*}
The statement follows from these recursions by induction on~$r$.
\end{proof}

\subsection{Classification}

We next describe the Harish-Chandra modules~$\varpi(f,k)$ associated with harmonic weak Maa\ss\ forms in terms of the standard modules~$\varpi^\pm(\pm\nu)$. The theory is much more stringent than in the case of integral weights.

\begin{proposition}
Let~$f$ be a weakly holomorphic modular form of weight~$k \in \frac{1}{2} + \ZZ$. Then~$\varpi(f,k)$ is isomorphic to~$\varpi^+(k-1)$.

Let~$f$ be a harmonic weak Maa\ss\ form of weight~$k \in \frac{1}{2} + \ZZ$ that is not weakly holomorphic. Then~$\varpi(f,k)$ fits into the nonsplit exact sequence
\begin{gather*}
  0
\lra
  \varpi^-(1-k)
\lra
  \varpi(f,k)
\lra
  \varpi^+(k-1)
\lra
  0
\tx{.}
\end{gather*}
\end{proposition}
\begin{remark}
The existence of weakly holomorphic modular forms in all half-integral weight cases is clear. The existence of harmonic weak Maa\ss\ forms in all weights follows along the lines of Bruinier-Funke~\cite{bf2004}, when removing the condition that the image under~$\xi_k$ has moderate growth.
\end{remark}
\begin{proof}
One can appeal to a classification that identifies irreducible Harish-Chandra modules in terms of their eigenvalues under the Casimir element and their $K$\nbd types support.

A more elementary approach was suggested by Bringmann-Kudla~ \cite{bringmannkudla}. Since all $K$\nbd types in~$\varpi(f,k)$ appear with multiplicity at most one, it suffices to compare the action of~$X_\pm^r X_\mp^r$ and~$X_\mp^r X_\pm^r$. We give the details in the case that~$f$ is not weakly holomorphic.

Observe that~$g := \rmL_k\,f$ is annihilated by~$\rmR_{k-2}$. Using the intertwining property~\eqref{eq:Xpm-RL-intertwining} of~$\rmR_{k-2}$ and~$X_+$, we conclude that~$X_+\, \td{g}_{k-2} = 0$. To show that the Harish-Chandra-module~$\varpi(g, k-2)$ is isomorphic to~$\varpi^-(1-k)$, it now suffices to calculate and compare the eigenvalues of~$\td{g} = \td{f}_{-1}$ and of~$\phi_{k-2}$ from Section~\ref{ssec:principal-series} under~$X_+^r X_-^r$ for all positive integers~$r$. The former was given in~\eqref{eq:tdf-XmpXpm-eigenvalues} and the latter in~\eqref{eq:phij-Xmp-Xpm}, where we have~$\nu = k-2$ and~$j = k-2$.

The $K$\nbd type support of the quotient module~$\varpi(f,k) \slash \varpi(g,k-2)$ corresponds to the $k(\theta)$-eigenvalues~$e(i j \theta)$ for~$j \in k + 2 \ZZ_{\ge 0}$, which coincides with the $K$\nbd types that appear in~$\varpi^+(k-1)$. Again because each~$K$\nbd type occurs with multiplicity one, it suffices to compare the eigenvalues of~$\td{f}$ and~$\phi_k$ under~$X_-^r X_+^r$. The former was also given in~\eqref{eq:tdf-XmpXpm-eigenvalues} and the latter in~\eqref{eq:phij-Xmp-Xpm}, where we have~$\nu = k-2$ and~$j = k$.
\end{proof}

\subsection{Diagrams of $K$-types}\label{subsec:diagrams}

Let~$f$ be a harmonic weak Maa\ss\ form of weight~$k$. If we have~$\rmL_k\, f = 0$ and $k \in \frac{1}{2} + 2\ZZ$ is greater than one, the associated Harish-Chandra module~$\varpi(f,k)$ yields the~$K$\nbd type diagram
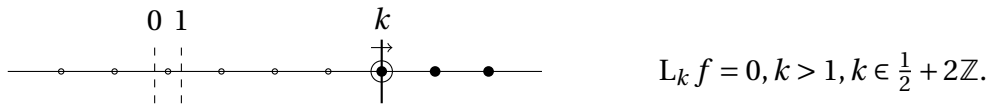
\begin{center}
\begin{tikzpicture}
\coordinate (o) at (0pt,0pt); %
\coordinate (s) at (10pt,0pt); %
\coordinate (vh) at (0pt,12pt); %

\coordinate (lb) at ($(o)-2*(s)$); 
\coordinate (le) at ($(o)+18*(s)$);
\draw[-] (lb) -- (le);

\foreach \ix in {0,2}{
  \coordinate (a) at ($(o)+\ix*(s)$);
  \draw (a) circle[radius=1pt];
}

\foreach \ix in {4,6,8,10}{
  \coordinate (a) at ($(o)+\ix*(s)$);
  \draw (a) circle[radius=1pt];
}

\foreach \ix in {12,14,16}{
  \coordinate (a) at ($(o)+\ix*(s)$);
  \fill (a) circle[radius=2pt];
}

\coordinate (a) at ($(o)+12*(s)$);
\draw (a) circle[radius=4pt];

\coordinate (a) at ($(o)+3.5*(s)$);
\coordinate (at) at ($(a)+(vh)$);
\coordinate (ab) at ($(a)-(vh)$);
\draw[dashed] (ab) -- (at) node[above] {$0$};

\coordinate (a) at ($(o)+4.5*(s)$);
\coordinate (at) at ($(a)+(vh)$);
\coordinate (ab) at ($(a)-(vh)$);
\draw[dashed] (ab) -- (at) node[above] {$1$};

\coordinate (a) at ($(o)+12*(s)$);
\coordinate (at) at ($(a)+(vh)$);
\coordinate (ab) at ($(a)-(vh)$);
\draw[-,thick] (ab) -- (at) node[above] {$k$};

\coordinate (b) at ($(a)+0.75*(vh)$);
\coordinate (bl) at ($(b)-0.4*(s)$);
\coordinate (br) at ($(b)+0.4*(s)$);
\draw[->] (bl) -- (br);

\coordinate (a) at ($22*(s)$);
\node[right] at (a)
  {$\rmL_k\, f = 0\tx{,}\, k > 1\tx{,}\, k \in \frac{1}{2} + 2\ZZ\tx{.}$};
\end{tikzpicture}%
\end{center}
Observe that the $K$\nbd types next to~$0$ in this diagram are labelled by~$-\frac{3}{2}$ and~$\frac{1}{2}$. Integral weights do not support~$K$\nbd types in this diagram. If we have~$\rmL_k\, f = 0$ and $k \in \frac{3}{2} + 2\ZZ$ is greater than one, then the positively labelled $K$-type next to zero is~$\frac{3}{2}$ and the negative one is~$-\frac{1}{2}$. This yields the diagram
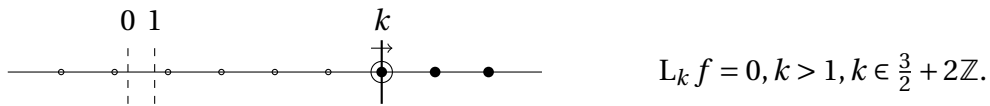
\begin{center}
\begin{tikzpicture}
\coordinate (o) at (0pt,0pt); %
\coordinate (s) at (10pt,0pt); %
\coordinate (vh) at (0pt,12pt); %

\coordinate (lb) at ($(o)-2*(s)$); 
\coordinate (le) at ($(o)+18*(s)$);
\draw[-] (lb) -- (le);

\foreach \ix in {0,2}{
  \coordinate (a) at ($(o)+\ix*(s)$);
  \draw (a) circle[radius=1pt];
}

\foreach \ix in {4,6,8,10}{
  \coordinate (a) at ($(o)+\ix*(s)$);
  \draw (a) circle[radius=1pt];
}

\foreach \ix in {12,14,16}{
  \coordinate (a) at ($(o)+\ix*(s)$);
  \fill (a) circle[radius=2pt];
}

\coordinate (a) at ($(o)+12*(s)$);
\draw (a) circle[radius=4pt];

\coordinate (a) at ($(o)+2.5*(s)$);
\coordinate (at) at ($(a)+(vh)$);
\coordinate (ab) at ($(a)-(vh)$);
\draw[dashed] (ab) -- (at) node[above] {$0$};

\coordinate (a) at ($(o)+3.5*(s)$);
\coordinate (at) at ($(a)+(vh)$);
\coordinate (ab) at ($(a)-(vh)$);
\draw[dashed] (ab) -- (at) node[above] {$1$};

\coordinate (a) at ($(o)+12*(s)$);
\coordinate (at) at ($(a)+(vh)$);
\coordinate (ab) at ($(a)-(vh)$);
\draw[-,thick] (ab) -- (at) node[above] {$k$};

\coordinate (b) at ($(a)+0.75*(vh)$);
\coordinate (bl) at ($(b)-0.4*(s)$);
\coordinate (br) at ($(b)+0.4*(s)$);
\draw[->] (bl) -- (br);

\coordinate (a) at ($22*(s)$);
\node[right] at (a)
  {$\rmL_k\, f = 0\tx{,}\, k > 1\tx{,}\, k \in \frac{3}{2} + 2\ZZ\tx{.}$};
\end{tikzpicture}%
\end{center}

If~$\rmL_k\, f \ne 0$, we again have two cases. One of the following two diagrams describes the~$K$\nbd types in~$\varpi(f,k)$:
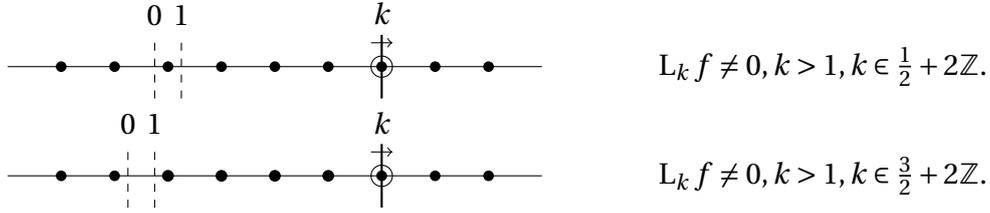
\begin{center}
\begin{tikzpicture}
\coordinate (o) at (0pt,0pt); %
\coordinate (s) at (10pt,0pt); %
\coordinate (vh) at (0pt,12pt); %

\coordinate (lb) at ($(o)-2*(s)$); 
\coordinate (le) at ($(o)+18*(s)$);
\draw[-] (lb) -- (le);

\foreach \ix in {0,2}{
  \coordinate (a) at ($(o)+\ix*(s)$);
  \fill (a) circle[radius=2pt];
}

\foreach \ix in {4,6,8,10}{
  \coordinate (a) at ($(o)+\ix*(s)$);
  \fill (a) circle[radius=2pt];
}

\foreach \ix in {12,14,16}{
  \coordinate (a) at ($(o)+\ix*(s)$);
  \fill (a) circle[radius=2pt];
}

\coordinate (a) at ($(o)+12*(s)$);
\draw (a) circle[radius=4pt];

\coordinate (a) at ($(o)+3.5*(s)$);
\coordinate (at) at ($(a)+(vh)$);
\coordinate (ab) at ($(a)-(vh)$);
\draw[dashed] (ab) -- (at) node[above] {$0$};

\coordinate (a) at ($(o)+4.5*(s)$);
\coordinate (at) at ($(a)+(vh)$);
\coordinate (ab) at ($(a)-(vh)$);
\draw[dashed] (ab) -- (at) node[above] {$1$};

\coordinate (a) at ($(o)+12*(s)$);
\coordinate (at) at ($(a)+(vh)$);
\coordinate (ab) at ($(a)-(vh)$);
\draw[-,thick] (ab) -- (at) node[above] {$k$};

\coordinate (b) at ($(a)+0.75*(vh)$);
\coordinate (bl) at ($(b)-0.4*(s)$);
\coordinate (br) at ($(b)+0.4*(s)$);
\draw[->] (bl) -- (br);

\coordinate (a) at ($22*(s)$);
\node[right] at (a)
  {$\rmL_k\, f \ne 0\tx{,}\, k > 1\tx{,}\, k \in \frac{1}{2} + 2\ZZ\tx{.}$};
\end{tikzpicture}%

\begin{tikzpicture}
\coordinate (o) at (0pt,0pt); %
\coordinate (s) at (10pt,0pt); %
\coordinate (vh) at (0pt,12pt); %

\coordinate (lb) at ($(o)-2*(s)$); 
\coordinate (le) at ($(o)+18*(s)$);
\draw[-] (lb) -- (le);

\foreach \ix in {0,2}{
  \coordinate (a) at ($(o)+\ix*(s)$);
  \fill (a) circle[radius=2pt];
}

\foreach \ix in {4,6,8,10}{
  \coordinate (a) at ($(o)+\ix*(s)$);
  \filldraw (a) circle[radius=2pt];
}

\foreach \ix in {12,14,16}{
  \coordinate (a) at ($(o)+\ix*(s)$);
  \fill (a) circle[radius=2pt];
}

\coordinate (a) at ($(o)+12*(s)$);
\draw (a) circle[radius=4pt];

\coordinate (a) at ($(o)+2.5*(s)$);
\coordinate (at) at ($(a)+(vh)$);
\coordinate (ab) at ($(a)-(vh)$);
\draw[dashed] (ab) -- (at) node[above] {$0$};

\coordinate (a) at ($(o)+3.5*(s)$);
\coordinate (at) at ($(a)+(vh)$);
\coordinate (ab) at ($(a)-(vh)$);
\draw[dashed] (ab) -- (at) node[above] {$1$};

\coordinate (a) at ($(o)+12*(s)$);
\coordinate (at) at ($(a)+(vh)$);
\coordinate (ab) at ($(a)-(vh)$);
\draw[-,thick] (ab) -- (at) node[above] {$k$};

\coordinate (b) at ($(a)+0.75*(vh)$);
\coordinate (bl) at ($(b)-0.4*(s)$);
\coordinate (br) at ($(b)+0.4*(s)$);
\draw[->] (bl) -- (br);

\coordinate (a) at ($22*(s)$);
\node[right] at (a)
  {$\rmL_k\, f \ne 0\tx{,}\, k > 1\tx{,}\, k \in \frac{3}{2} + 2\ZZ\tx{.}$};
\end{tikzpicture}%
\end{center}

The situation is very similar for~$k$ less than~$1$. Depending on~$\rmL_k\,f$ and~$k$, one of the following four diagrams arises from~$\varpi(f,k)$:
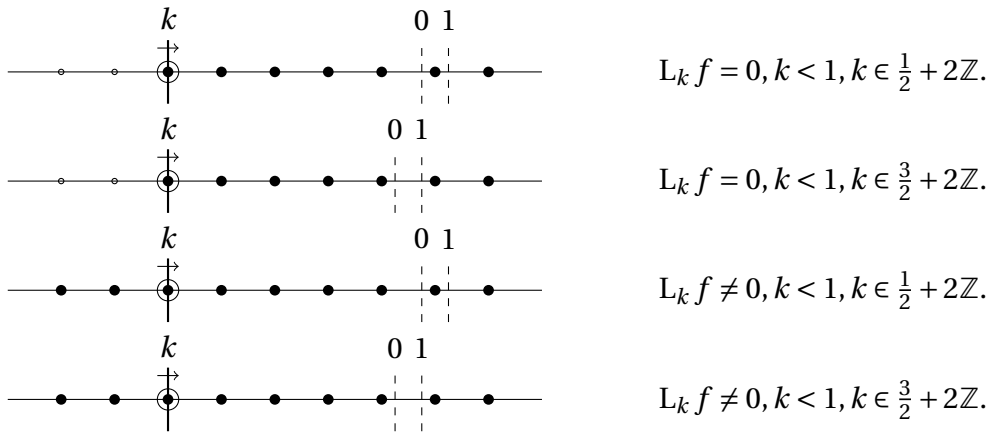
\begin{center}
\begin{tikzpicture}
\coordinate (o) at (0pt,0pt); %
\coordinate (s) at (10pt,0pt); %
\coordinate (vh) at (0pt,12pt); %

\coordinate (lb) at ($(o)-2*(s)$); 
\coordinate (le) at ($(o)+18*(s)$);
\draw[-] (lb) -- (le);

\foreach \ix in {0,2}{
  \coordinate (a) at ($(o)+\ix*(s)$);
  \draw (a) circle[radius=1pt];
}

\foreach \ix in {4,6,8,10,12}{
  \coordinate (a) at ($(o)+\ix*(s)$);
  \fill (a) circle[radius=2pt];
}

\foreach \ix in {14,16}{
  \coordinate (a) at ($(o)+\ix*(s)$);
  \fill (a) circle[radius=2pt];
}

\coordinate (a) at ($(o)+4*(s)$);
\draw (a) circle[radius=4pt];

\coordinate (a) at ($(o)+4*(s)$);
\coordinate (at) at ($(a)+(vh)$);
\coordinate (ab) at ($(a)-(vh)$);
\draw[-,thick] (ab) -- (at) node[above] {$k$};

\coordinate (b) at ($(a)+0.75*(vh)$);
\coordinate (bl) at ($(b)-0.4*(s)$);
\coordinate (br) at ($(b)+0.4*(s)$);
\draw[->] (bl) -- (br);

\coordinate (a) at ($(o)+13.5*(s)$);
\coordinate (at) at ($(a)+(vh)$);
\coordinate (ab) at ($(a)-(vh)$);
\draw[dashed] (ab) -- (at) node[above] {$0$};

\coordinate (a) at ($(o)+14.5*(s)$);
\coordinate (at) at ($(a)+(vh)$);
\coordinate (ab) at ($(a)-(vh)$);
\draw[dashed] (ab) -- (at) node[above] {$1$};

\coordinate (a) at ($22*(s)$);
\node[right] at (a)
  {$\rmL_k\, f = 0\tx{,}\, k < 1\tx{,}\, k \in \frac{1}{2} + 2\ZZ\tx{.}$};
\end{tikzpicture}%

\begin{tikzpicture}
\coordinate (o) at (0pt,0pt); %
\coordinate (s) at (10pt,0pt); %
\coordinate (vh) at (0pt,12pt); %

\coordinate (lb) at ($(o)-2*(s)$); 
\coordinate (le) at ($(o)+18*(s)$);
\draw[-] (lb) -- (le);

\foreach \ix in {0,2}{
  \coordinate (a) at ($(o)+\ix*(s)$);
  \draw (a) circle[radius=1pt];
}

\foreach \ix in {4,6,8,10,12}{
  \coordinate (a) at ($(o)+\ix*(s)$);
  \fill (a) circle[radius=2pt];
}

\foreach \ix in {14,16}{
  \coordinate (a) at ($(o)+\ix*(s)$);
  \fill (a) circle[radius=2pt];
}

\coordinate (a) at ($(o)+4*(s)$);
\draw (a) circle[radius=4pt];

\coordinate (a) at ($(o)+4*(s)$);
\coordinate (at) at ($(a)+(vh)$);
\coordinate (ab) at ($(a)-(vh)$);
\draw[-,thick] (ab) -- (at) node[above] {$k$};

\coordinate (b) at ($(a)+0.75*(vh)$);
\coordinate (bl) at ($(b)-0.4*(s)$);
\coordinate (br) at ($(b)+0.4*(s)$);
\draw[->] (bl) -- (br);

\coordinate (a) at ($(o)+12.5*(s)$);
\coordinate (at) at ($(a)+(vh)$);
\coordinate (ab) at ($(a)-(vh)$);
\draw[dashed] (ab) -- (at) node[above] {$0$};

\coordinate (a) at ($(o)+13.5*(s)$);
\coordinate (at) at ($(a)+(vh)$);
\coordinate (ab) at ($(a)-(vh)$);
\draw[dashed] (ab) -- (at) node[above] {$1$};

\coordinate (a) at ($22*(s)$);
\node[right] at (a)
  {$\rmL_k\, f = 0\tx{,}\, k < 1\tx{,}\, k \in \frac{3}{2} + 2\ZZ\tx{.}$};
\end{tikzpicture}%

\begin{tikzpicture}
\coordinate (o) at (0pt,0pt); %
\coordinate (s) at (10pt,0pt); %
\coordinate (vh) at (0pt,12pt); %

\coordinate (lb) at ($(o)-2*(s)$); 
\coordinate (le) at ($(o)+18*(s)$);
\draw[-] (lb) -- (le);

\foreach \ix in {0,2}{
  \coordinate (a) at ($(o)+\ix*(s)$);
  \fill (a) circle[radius=2pt];
}

\foreach \ix in {4,6,8,10,12}{
  \coordinate (a) at ($(o)+\ix*(s)$);
  \fill (a) circle[radius=2pt];
}

\foreach \ix in {14,16}{
  \coordinate (a) at ($(o)+\ix*(s)$);
  \fill (a) circle[radius=2pt];
}

\coordinate (a) at ($(o)+4*(s)$);
\draw (a) circle[radius=4pt];

\coordinate (a) at ($(o)+4*(s)$);
\coordinate (at) at ($(a)+(vh)$);
\coordinate (ab) at ($(a)-(vh)$);
\draw[-,thick] (ab) -- (at) node[above] {$k$};

\coordinate (b) at ($(a)+0.75*(vh)$);
\coordinate (bl) at ($(b)-0.4*(s)$);
\coordinate (br) at ($(b)+0.4*(s)$);
\draw[->] (bl) -- (br);

\coordinate (a) at ($(o)+13.5*(s)$);
\coordinate (at) at ($(a)+(vh)$);
\coordinate (ab) at ($(a)-(vh)$);
\draw[dashed] (ab) -- (at) node[above] {$0$};

\coordinate (a) at ($(o)+14.5*(s)$);
\coordinate (at) at ($(a)+(vh)$);
\coordinate (ab) at ($(a)-(vh)$);
\draw[dashed] (ab) -- (at) node[above] {$1$};

\coordinate (a) at ($22*(s)$);
\node[right] at (a)
  {$\rmL_k\, f \ne 0\tx{,}\, k < 1\tx{,}\, k \in \frac{1}{2} + 2\ZZ\tx{.}$};
\end{tikzpicture}%

\begin{tikzpicture}
\coordinate (o) at (0pt,0pt); %
\coordinate (s) at (10pt,0pt); %
\coordinate (vh) at (0pt,12pt); %

\coordinate (lb) at ($(o)-2*(s)$); 
\coordinate (le) at ($(o)+18*(s)$);
\draw[-] (lb) -- (le);

\foreach \ix in {0,2}{
  \coordinate (a) at ($(o)+\ix*(s)$);
  \fill (a) circle[radius=2pt];
}

\foreach \ix in {4,6,8,10,12}{
  \coordinate (a) at ($(o)+\ix*(s)$);
  \fill (a) circle[radius=2pt];
}

\foreach \ix in {14,16}{
  \coordinate (a) at ($(o)+\ix*(s)$);
  \fill (a) circle[radius=2pt];
}

\coordinate (a) at ($(o)+4*(s)$);
\draw (a) circle[radius=4pt];

\coordinate (a) at ($(o)+4*(s)$);
\coordinate (at) at ($(a)+(vh)$);
\coordinate (ab) at ($(a)-(vh)$);
\draw[-,thick] (ab) -- (at) node[above] {$k$};

\coordinate (b) at ($(a)+0.75*(vh)$);
\coordinate (bl) at ($(b)-0.4*(s)$);
\coordinate (br) at ($(b)+0.4*(s)$);
\draw[->] (bl) -- (br);

\coordinate (a) at ($(o)+12.5*(s)$);
\coordinate (at) at ($(a)+(vh)$);
\coordinate (ab) at ($(a)-(vh)$);
\draw[dashed] (ab) -- (at) node[above] {$0$};

\coordinate (a) at ($(o)+13.5*(s)$);
\coordinate (at) at ($(a)+(vh)$);
\coordinate (ab) at ($(a)-(vh)$);
\draw[dashed] (ab) -- (at) node[above] {$1$};

\coordinate (a) at ($22*(s)$);
\node[right] at (a)
  {$\rmL_k\, f \ne 0\tx{,}\, k < 1\tx{,}\, k \in \frac{3}{2} + 2\ZZ\tx{.}$};
\end{tikzpicture}%
\end{center}

\section{Theta lifts of harmonic weak Maa\ss\ forms}\label{sec:lifts}
In this section we review results on the extension of the classical Shintani lifting to harmonic weak Maa\ss\ forms and the so-called Millson theta lifting of such forms obtained by the first named author in joint work with Markus Schwagenscheidt \cite{alfesschwagenscheidtmillson, alfesschwagenscheidtshintani}. These liftings are an explicit realization of the theta correspondence given as the integral of an input function transforming like a modular form of (even) weight $k$ against a certain theta kernel function. 
We illustrate this procedure in a bit more detail. If $f$ transforms of weight $k$ one is led to consider the following integral (where the integration is carried over a suitable fundamental domain $\mathcal{F}$)
\[
\int_\mathcal{F} f(z) \overline{\Theta(\varphi, \tau, z)} \Im(z)^{-k} \frac{dxdy}{y^2}, \, z=x+iy.
\]
Here, $\Theta(\varphi, \tau, z)$ is an integration kernel of weight $k$ in the variable $z$. In the variable $\tau$ the complex conjugate $\overline{\Theta(\varphi, \tau, z)}$ is of half-integral weight $\ell$. Moreover, $\varphi$ is a suitable Schwartz function, Provided the integral converges, it transforms like an automorphic form of weight $\ell$.

For holomorphic modular forms such liftings have been investigated in the framework of the Shimura-Shintani-correspondence \cite{sh,shintani,kohnen80,kohnen82}. Ideas of Harvey and Moore \cite{harveymoore} and Bor\-cherds \cite{borcherds} led to the theory of regularized theta liftings allowing for inputs that are not holomorphic at the cusps. 
The lifts we consider in this work
serve as generating series of traces of CM values and (regularized) geodesic cycle integrals of the input function.

We will consider twisted versions of the Shintani and Millson lift. These are obtained via twisting the theta kernel with a certain genus character (see \cite{ae} for a description of this procedure). This enables us to state our results for the full modular group. In particular, we let $\Delta\in \mathbb{Z}$ be a fundamental discriminant. 

We denote the Millson theta lift by $\IM$ and the Shintani theta lift by $\ISh$.

\begin{remark}
We state the results in the following subsections for theta lifts of forms for the full modular group to half-integral weight forms for the group $\Gamma_0(4)$. Note that their results hold in greater generality (see \cite{alfesschwagenscheidtmillson, alfesschwagenscheidtshintani}). 
\end{remark}

\subsection{The Shintani lifting}

In the past years the classical Shintani theta lift of holomorphic forms has been generalized to weakly holomorphic modular forms by Bringmann, Guerzhoy and Kane \cite{briguka2,briguka} and to differentials of the third kind by Bruinier, Funke, Imamoglu and Li \cite{bifl}. In \cite{alfesschwagenscheidtshintani} the first author considered the Shintani lift of harmonic weak Maa\ss\ forms together with Schwagenscheidt. 

Their results can be summarized as follows. We do not state the explicit Fourier expansion (since we do not need it in the course of this paper). The Fourier coefficients of the holomorphic part are given by the regularized traces of geodesic cycle integrals of the integral weight form. 
\begin{theorem}[Proposition 5.2 and Theorem 6.1 in \cite{alfesschwagenscheidtshintani}]\label{thm:shin}
Consider $k\in\Z_{\geq 0}$ such that $(-1)^{k+1}\Delta >0$.
The regularized Shintani theta lift $\ISh_{\Delta}(G,\tau)$ of a harmonic Maa\ss\ form $G \in H_{2k+2}$ exists and defines a harmonic Maa\ss\ form in $H_{3/2+k}$. If\/ $G \in M_{2k+2}^{!}$ is a weakly holomorphic modular form then $\ISh_{\Delta}(G,\tau) \in M_{3/2+k}$ is a holomorphic modular form, and if in addition $a_{G}^{+}(0) = 0$ then $\ISh_{\Delta}(G,\tau) \in S_{3/2+k}$ is a cusp form.
\end{theorem}

\subsection{The Millson theta lifting}

In \cite{zagiertraces} Zagier considered traces of the values of the modular invariant $j(z)$ at quadratic irrationalities. He showed that these traces are the Fourier coefficients of modular forms of half-integral weight (both of weight $1/2$ and $3/2$). %

Using the framework of \cite{funkephd} Bruinier and Funke \cite{brfu06} showed that such modularity results in weight $3/2$ for generating series of traces of modular functions can be obtained via the Kudla-Millson theta lift. Their work was generalized in various directions: to twisted traces in \cite{ae}, to higher weight in \cite{brono2} and \cite{alfes}. In \cite{agor} and \cite{alfesschwagenscheidtmillson} a different theta lift, the so-called Millson theta lift, was considered which then fully recovered Zagier's results. %

We now briefly review the results of \cite{alfesschwagenscheidtmillson}.
We remark that the Fourier coefficients of the holomorphic part are given by the traces of CM values of a suitable derivative of the input.
\begin{theorem}[Theorem 1.1 in \cite{alfesschwagenscheidtmillson} and Proposition 5.5 in \cite{alfesschwagenscheidtshintani}]\label{thm:millson}
	Let $k \in \Z_{\geq 0}$ such that $(-1)^{k}\Delta < 0$ and let $F \in H_{-2k}^{\mathrm{mg}}$. 
	\begin{enumerate} \item Let $k\neq 0$. 
The Millson theta lift $\IM_{\Delta}(F,\tau)\in H_{1/2-k}^{\mathrm{mg}}$ is a harmonic weak Maa\ss\ form of weight ${1/2-k}$ for $\Gamma_{0}(4)$ satisfying the Kohnen plus space condition. Further, if $F$ is weakly holomorphic, then so is $\IM_{\Delta}(F,\tau)$.
\item  Let $k=0$ and let $F$ be such that the constant term of its non-holomorphic part vanishes. Then the Millson theta lift $\IM_{\Delta}(F,\tau)\in H^{\mathrm{mg}}_{1/2}$ is a harmonic weak Maa\ss\ form of weight ${1/2}$ for $\Gamma_{0}(4)$ satisfying the Kohnen plus space condition. 

\end{enumerate}
\end{theorem}

\begin{remark}
The Millson and Shintani lifting are related via a differential equation satisfied by the two theta kernels:
\begin{align*}
	\xi_{1/2-k,\tau}\IM_{\Delta}(F,\tau)& = -4^{k}\sqrt{|\Delta|}\ISh_{\Delta}(\xi_{-2k,z}F,\tau)\\
		\xi_{3/2+k,\tau}\ISh_{\Delta}(G,\tau) &= -\frac{1}{4^{k+1}\sqrt{|\Delta|}}\IM_{\Delta}(\xi_{2k+2,z}G,\tau).
	\end{align*}
\end{remark}

\section{Examples}

In this section we give explicit examples for each of the cases occurring in Section~\ref{subsec:diagrams}. These are given as the Millson theta lifts of forms of weight $-2k\leq 0$ and Shintani theta lifts of forms of weight $2k\geq 2$. We denote the integral weights by $2k\in 2\ZZ$ and the half-integral weights by $\ell\in \frac12\ZZ\setminus \ZZ$. 

\begin{remark}
To realize the cases associated to half-integral weight it suffices to consider lifts of the scalar-valued examples that occur in \cite{bringmannkudla}. Nonetheless, it would be interesting to extend the theory of theta liftings to symmetric power types along the lines of Funke and Millson's work \cite{funkemillson}.
\end{remark}

\subsection{Weight $\ell\leq \frac12$}
We first let $\ell\leq \frac12$ and need to provide functions $f\in H_{\ell=1/2-k}^{\text{mg}}$ that satisfy $\rmL_\ell f=0$ and $\rmL_\ell f\neq 0$. 

\subsubsection{Weight $\ell\leq \frac12$ and $\rmL\ell f=0$}
First note that the Millson lift of the constant function, that gives an example for case I (a) in \cite{bringmannkudla} (i.e., it satisfies $\rmL_0 f=0$ and $\rmR_0 f=0$), is a weakly holomorphic modular form of weight $\ell=1/2$ as can easily be deduced from Proposition 3.4.8 in \cite{markusdiss}.

If $\ell\leq -1/2$, we can take the Millson lift of a weakly holomorphic modular form of weight $-2k<0$ (compare case I (b) in \cite{bringmannkudla}). We see from Theorem \ref{thm:millson} that the Millson lift of a weakly holomorphic modular form $F\in M^!_{-2k}$ is again weakly holomorphic of weight $\ell=1/2-k$ if $-2k<0$. 

\subsubsection{Weight $\ell\leq \frac12$ and $\rmL\ell f\neq 0$}
If $k=0$, the lift lies in the space of harmonic weak Maa\ss\ forms $H_{1/2}^{\mathrm{mg}}$. This already realizes the case of weight $\ell=1/2$ when we require $\rmL_\ell f\neq 0$. For $\ell\leq -1/2$ we consider a  function $f$ satisfying $\rmL_{-2k} f\neq 0$ and $\rmR_{-2k}^{1+2k} f=0$ (corresponding to case I (c) in \cite{bringmannkudla}). We lift the realization of \cite{bringmannkudla}: We let $F\in M_{-2k}^!\setminus\{0\}$ and take $G:=\calF_{-2k}F$. Note that if 
\[
F(z)=\sum_{n\gg-\infty} c_F^+(n)q^n\in M^!_{-2k},
\]
then, compare \eqref{eq:feflip}, we have
\[
\mathcal{F}_{-2k}(F(z))= -\overline{c_F^+(0)} - \frac{1}{(2k)!} \sum_{\substack{n\ll\infty\\n\neq 0}} \overline{c_F^+(-n)} \Gamma(1+2k, -4\pi n v)q^n.
\]
From Theorem \ref{thm:millson} we easily deduce that the lift of $G$ is in $H^{\mathrm{mg}}_{1/2-k}$.

\begin{remark}
For the sake of completeness we explain the Millson lift of the remaining cases that Bringmann and Kudla consider. Their case IV (d) $\rmL_{-2k} f\neq 0$ and $\rmR_{-2k}^{1+2k} f\neq 0$ is realized by letting $F\in M_{-2k}^!\setminus\{0\}$ and taking $G:=F+\calF_{-2k}F$. Lifting this we obviously obtain a combination of the previous two cases.

Moreover, we note that the lift of the Eisenstein series 
is again an Eisenstein series of weight $1/2-k$. This can be shown by standard arguments (for example using work of Crawford and Funke\cite{funkecrawford}).
\end{remark}

\begin{remark}
We remark that the weight $k=0$ ($\ell=1/2$) case can also be realized by the Siegel lift investigated in \cite{bif}.
\end{remark}

\subsection{Weight $\ell \geq \frac32$}

\subsubsection{Weight $\ell \geq \frac32$ and $\rmL_\ell f=0$}
Considering the Shintani lift of cusp forms of integral weight $2k\geq 2$ we see that these give us examples of harmonic weak Maa\ss\ forms of weight $\ell =3/2+k\geq 3/2$ satisfying $\rmL_\ell f=0$. Cusp forms of integral weight correspond to case III (a) in the classification of Bringmann and Kudla.

\subsubsection{Weight $\ell \geq \frac32$ and $\rmL_\ell f\neq 0$}
We can realize the case of $\rmL_\ell f\neq 0$ by taking the Shintani lift of the weight $2$ Eisenstein series (case III (b) in \cite{bringmannkudla}) and sesquiharmonic Poincar\'e series (case III (c) in \cite{bringmannkudla}). 

To give an example for a function $f$ with $\rmL_{3/2} f\neq0$, we consider the lift of the weight $2$ Eisenstein series
\[
E_2^*(z)= 1-24\sum_{n \geq 1}\sigma_{1}(n)e^{2\pi i n z} - \frac{3}{\pi y}.
\]
It was computed in \cite{alfesschwagenscheidtshintani}. We have
	\[
	\sqrt{|\Delta|}\ISh_{\Delta}(E_{2}^{*},\tau) = 12H(|\Delta|)E_{3/2}^{*}(\tau),
	\]
	where
\begin{align*}
	E_{3/2}^{*}(\tau) = \sum_{D \geq 0}H(D)e^{2\pi i D\tau}+\frac{1}{16\pi}\sum_{n \in \Z}v^{-1/2}\beta_{3/2}(4\pi n^{2}v)e^{-2\pi i n^{2}\tau},
	\end{align*}
	with $H(0) = -\frac{1}{12}$ and $H(D) = 0$ if $-D \neq 0$ is not a discriminant, is Zagier's weight $3/2$ Eisenstein series (see \cite{zagiereisenstein}). Moreover, $\beta_{3/2}(s)$ is defined as in the introduction.

The third case in \cite{bringmannkudla} is characterized by $\rmL_{2k} f \neq 0$ and $\rmL_{2k}^{2k} f\neq 0$. An example is constructed via certain sesquiharmonic Poincar\'e series that are in fact harmonic (this relies on the vanishing of the dual space of cusp forms). We do not explicitly compute the lift of such series but note that according to Theorem \ref{thm:shin} the lift is a harmonic weak Maa\ss\ form of moderate growth of weight $\ell=3/2+k$.

\begin{remark}
We remark that the weight $3/2$ case can also be realized as the Kudla-Millson lift of a harmonic weak Maa\ss\ form of weight $0$ as in \cite{brfu06}.
\end{remark}

\ifbool{nobiblatex}{%
  \bibliographystyle{alpha}%
  \bibliography{bib.bib}%
}{%
  \renewcommand{\baselinestretch}{.8}
  \Needspace*{4em}
  \printbibliography[heading=none]%
}

\Needspace*{3\baselineskip}
\noindent
\rule{\textwidth}{0.15em}

{\noindent\small
Universität Bielefeld, Fakultät für Mathematik, Postfach 100 131, 33501 Bielefeld, Germany\\
E-Mail: \url{alfes@math.uni-bielefeld.de}\\

\noindent
Chalmers tekniska högskola och G\"oteborgs Universitet,
Institutionen för Matematiska vetenskaper,
SE-412 96 Göteborg, Sweden\\
E-mail: \url{martin@raum-brothers.eu}\\
Homepage: \url{http://raum-brothers.eu/martin}

}%

\end{document}

